\documentclass[12pt]{amsart}
\usepackage{amsmath}
\usepackage{amsthm}
\usepackage{amssymb}
\usepackage{verbatim}
\usepackage{subfigure}
\usepackage{verbatim}
\usepackage{url}
\numberwithin{equation}{section}
\usepackage{textcomp}
\usepackage{stmaryrd}
\usepackage{xy}
\usepackage{color}
\usepackage{mathrsfs}
\usepackage{setspace}
\usepackage{enumerate}
\usepackage{amsthm}
\definecolor{Mygrey}{gray}{0.75}

\makeatletter
\def\@tocline#1#2#3#4#5#6#7{\relax
  \ifnum #1>\c@tocdepth 
  \else
    \par \addpenalty\@secpenalty\addvspace{#2}%
    \begingroup \hyphenpenalty\@M
    \@ifempty{#4}{%
      \@tempdima\csname r@tocindent\number#1\endcsname\relax
    }{%
      \@tempdima#4\relax
    }%
    \parindent\z@ \leftskip#3\relax \advance\leftskip\@tempdima\relax
    \rightskip\@pnumwidth plus4em \parfillskip-\@pnumwidth
    #5\leavevmode\hskip-\@tempdima
      \ifcase #1
       \or\or \hskip 1em \or \hskip 2em \else \hskip 3em \fi%
      #6\nobreak\relax
    \dotfill\hbox to\@pnumwidth{\@tocpagenum{#7}}\par
    \nobreak
    \endgroup
  \fi}
\makeatother
\usepackage{hyperref}

\def\bfa{{\bf a}}
\def\displayandname#1{\rlap{$\displaystyle\csname #1\endcsname$}%
                      \qquad \texttt{\char92 #1}}

\makeatletter
\def\url@leostyle{%
  \@ifundefined{selectfont}{\def\UrlFont{\sf}}{\def\UrlFont{\small\ttfamily}}}
\makeatother
\newtheorem{thm}{Theorem}[section]
\newtheorem{pro}[thm]{Proposition}
\newtheorem{lem}[thm]{Lemma}

\newtheorem{con}[thm]{Conjecture}

\theoremstyle{definition}
\newtheorem{df}[thm]{Definition}

\theoremstyle{remark}
\newtheorem{rem}[thm]{Remark}
\newtheorem{exa}[thm]{Example}

\newtheorem{nota}[thm]{Notation}
\newtheorem*{example*}{Example}

\newcommand\M{{\operatorname{M}}}

\newcommand\val{{\operatorname{val}}}
\newcommand\trop{{\operatorname{trop}}}

\newcommand\Ind{{\operatorname{Ind}}}
\newcommand\adj{\operatorname{adj}}

\newcommand\tr{{\operatorname{tr}}}
\newcommand\Id{{\operatorname{Id}}}
\begin{document}

\bibliographystyle{acm}

\title{Dependence of supertropical eigenspaces}

 \author[ Adi Niv]{Adi Niv}
\address{INRIA Saclay--\^Ile-de-France and CMAP. \'Ecole Polytechnique, Route de Saclay,~91128 ${\ \ }$ Palaiseau Cedex, France.}
\email{Adi.Niv@inria.fr}

\author[L. Rowen]{Louis Rowen}
\address{Department of Mathematics. Bar-Ilan University, Ramat-Gan 52900,Israel}
\email{rowen@macs.biu.ac.il}

  \thanks{The first author is sported by the French Chateaubriand grant and  INRIA postdoctoral fellowship}

\begin{abstract}
We  study the pathology that  causes  tropical eigenspaces of
distinct supertropical eigenvalues of a nonsingular matrix $A$, to
be dependent. We show that in lower dimensions the eigenvectors of
distinct eigenvalues are independent, as desired. The index set that
differentiates between  subsequent essential monomials of the
characteristic polynomial,  yields an eigenvalue $\lambda$, and
corresponds to the columns of the eigenmatrix $A+\lambda I$ from
which the eigenvectors are taken. We ascertain the cause for failure
in higher dimensions, and prove that  independence of the
eigenvectors is recovered in case a certain ``difference criterion''
holds, defined in terms of disjoint differences between index sets
of subsequent coefficients. We conclude by considering the
eigenvectors of the matrix $A^\nabla : = \frac 1{\det(A)}{\adj(A)}$ and
the connection of the independence question to generalized
eigenvectors.

\end{abstract}

\thispagestyle{myheadings}
\font\rms=cmr8
\font\its=cmti8
\font\bfs=cmbx8

\markright{}
\def\thepage{}
\maketitle

\section{Introduction}

Although supertropical matrix algebra as developed in
\cite{STMA,STMA2} follows the general lines of classical linear
algebra (i.e., a Cayley-Hamilton Theorem, correspondence between the
roots of the characteristic polynomial and eigenvalues, Kramer's
rule, etc.), one encounters the anomaly in~\cite[Remark 5.3 and
Theorem 5.6]{STMA2} of a matrix   whose supertropical
 eigenvalues are distinct but whose corresponding supertropical eigenspaces are
 dependent.
 In this paper we examine how this happens, and give a criterion for
 the supertropical eigenspaces to be
 dependent, which we call the \textbf{difference criterion}, cf.~Definition~\ref{difcrit} and Theorem~\ref{es-n}.
 A pathological  example (\ref{exa}) is studied in depth to show why
 the difference criterion is critical.
We resolve the difficulty in general in Theorem~\ref{es-n1} by
passing to powers of $A$ and  considering generalized supertropical
eigenspaces.

\subsection{The tropical algebra and related structures}\label{sec:pre}

We start by  discussing briefly the max-plus algebra, its
refinements, and their relevance to applications.

The use of the max-plus algebra in tropical mathematics was inspired
by the function~$\log_t$, as the base $t$ of the logarithm
approaches~$0$. In the literature, this structure is usually studied
via valuations (see ~\cite{ST&SV} and ~\cite{STVT}) over the
field~$K=\mathbb{C}\{\{t\}\}$ of Puiseux series with powers
in~$\mathbb{Q}$ (resp.~$\mathbb{R}$, to the ordered
group~$(\mathbb{Q},+,\geq )$ (resp.~$(\mathbb{R},+,\geq )$). The
valuation is given by the lowest exponent appearing nontrivially in
the series (indeed~$v(ab)=v(a)+v(b)$ and~$v(a+b)\geq
\min(v(a),v(b))$). Then, we look at the dual structure obtained by
defining~$\trop(a)=-\val(a)$ and denoted as the tropicalization
of~$a\in K$. By setting~$\trop(a+b)$ to
be~$\max\{\trop(a),\trop(b)\}$, it is obvious that the tropical
structure deals with the uncertainty of equality in the valuation,
in the form of~$\trop(a+a)=\trop(a)$ (also equals to $\trop(-a)$).

\vskip 0.5 truecm

\subsection{The max-plus  algebra}

\vskip 0.1 truecm

The \textbf{tropical max-plus semifield} is an ordered group
$\mathcal{T}$ (usually the additive group of real
numbers~$\mathbb{R}$ or the set of rational numbers~$\mathbb{Q}$),
together with a formal element~$-\infty$ adjoined. The ordered group
$\mathcal{T}$ is made into a semiring equipped with the operations
$$a\varoplus b=\max\{a,b\}\ \text{ and }\ a\varodot b=a+b,$$ denoted
here as~$a+b$ and~$ab$ respectively (see ~\cite{MPA}, ~\cite{TAG}
and ~\cite{TA}). The unit element $1_{\mathcal{T}}$ is really the
element~$0\in \mathbb{Q}$, and $-\infty$ serves as the zero element.
\vskip 0.35 truecm

 Tropicalization enables one to simplify non-linear questions by putting them into a linear setting (see~\cite{PA}),
 which can be applied to discrete mathematics (see ~\cite{DMP}), optimization (see ~\cite{LOP}) and algebraic geometry (see ~\cite{TAG}).

In ~\cite{TSPM} Gaubert and Sharify introduce a general scaling technique, based on tropical algebra, which applies in particular to the companion form,  determining  the eigenvalues of a matrix polynomial. Akian, Gaubert  and Guterman show in ~\cite{polyhedra} that several decision problems originating from max-plus or tropical convexity are equivalent to zero-sum two player game problems.

\cite{T&I} is a collection of papers put together by Litvinov and
Sergeev. One main theme is the Maslov dequantization applied to
traditional mathematics over fields, built on the foundations of
idempotent analysis, tropical algebra, and tropical geometry.
Applications of idempotent mathematics were introduced by Litvinov
and Maslov in ~\cite{IM}.

On the side of pure mathematics, contributions are made in
~\cite{T&I} to idempotent analysis, tropical algebras, tropical
linear algebra and tropical convex geometry. Elaborate geometric
background with applications to  problems in classical (real and
complex) geometry can be found in ~\cite{TGA}. Here Mikhalkin viewed
the tropical structure as a branch of geometry manipulating with
certain piecewise-linear objects that take over the role of
classical algebraic varieties and describes hypersurfaces,
varieties, morphisms and moduli spaces in this setting.

Extensive mathematical applications have been made in combinatorics.
In this max-plus language, we may  use notions of linear algebra to
interpret combinatorial problems. In ~\cite{APP} Jonczy presents
some problems described by the Path algebra and solved by means of
$\min$ and $\max$ operations. Combinatorial overviews are given in
~\cite{CCP}, ~\cite{MA} of Butkovic  and ~\cite{ETCP} of Butkovic
and Murfitt, which focus on presenting a number of links between
basic max-algebraic problems  on the one hand and combinatorial
problems on the other hand. This indicates that the max-algebra may
be regarded as a linear-algebraic encoding of a class of
combinatorial problems.

\pagenumbering{arabic}
\addtocounter{page}{1}
\markboth{\SMALL ADI NIV AND LOUIS ROWEN}{\SMALL DEPENDENCE OF SUPERTROPICAL EIGENSPACES}

\vskip 0.45 truecm

\subsection{Supertropical algebra}

We pass to  the \textbf{supertropical semiring}, equipped with the
ghost ideal~$\mathcal{G}: = \mathcal{T}^{\nu}$, as established and
studied by~Izhakian and Rowen in ~\cite{STLA} and ~\cite{STA}.

We denote  as~$R=\mathcal{T}\cup \mathcal{G}\cup \{-\infty \}$ the
``standard''  supertropical semiring,  which contains the so-called
tangible elements of the structure and where we have a projection $R
\to \mathcal G$ given by $a \mapsto a^\nu$ for $a \in \mathcal T$
(and which is the identity map on $\mathcal G$). $\{a^{\nu}\in
\mathcal{G},\forall a\in \mathcal{T}\}$ are the ghost elements of
the structure, as defined in~\cite{STA}.
We write $0_R$
for $-\infty $, to stress its role as the zero element. On the one
hand,  $\mathcal{G}$ is a copy of the max-plus semifield, so $R$ can
be viewed as a cover of the max-plus semifield.

The supertropical semiring enables us to distinguish between a
maximal element~$a$ that is attained only once in a sum, i.e.,~$a\in
\mathcal{T}$ which is invertible, and a maximum that is being
attained at least twice, i.e.,~$a+a=a^{\nu}\in \mathcal{G}$, which
is not invertible. We do not distinguish between~$a+a$ and $a+a+a$
in this  structure. Note that $\nu$ projects the standard
supertropical semiring onto~$\mathcal{G}$, which can be identified
with the usual tropical structure. \vskip 0.35 truecm

%
%

In this new supertropical sense, we use the following order relation to describe two elements that are equal up to a ghost supplement:

\begin{df}\label{gs}
Let $a,b$ be any two elements in $R$. We say that $a$ \textbf{ghost
surpasses}~$b$, denoted $a\models_{gs} b$, if~$a=b+ghost$. That is,
$a=b\ $ or $a\in \mathcal{G}$ with $a^{\nu}\geq b^{\nu}$.

We say $a$ is~$\nu$-equivalent to~$b$, denoted by~$a\cong_{\nu}b$, if~$a^{\nu}=b^{\nu}$. That is, in the tropical structure,~$\nu$-equivalence projects to equality.
\vskip 0.15 truecm

\end{df}

    \noindent\textbf{Important properties of $\models_{gs}$:}
\begin{enumerate}
\item  $\models_{gs}$ is a partial order relation (see ~\cite[Lemma 1.5]{STMA2}).
\item If $a\models_{gs} b$ then $ac\models_{gs} bc$.
\item  If $a\models_{gs} b$ and $c\models_{gs} d$ then $a+c\models_{gs} b+d$ and $ac\models_{gs} bd$.
\item If $a\models_{gs} b$ and $a\in \mathcal{T}$, then $a=b$.
\end{enumerate}

Considering this relation, we regain basic algebraic properties that
were not accessible in the usual tropical setting, such as
multiplicativity of the tropical determinant, the  near
multiplicativity of the tropical adjoint, the role of roots in the
factorization of polynomials, the role of the determinant in matrix
singularity, a matrix that acts like an inverse,  common behavior of
similar matrices, classical properties of $ \adj(A),$ and the use of
elementary matrices. Tropical eigenspaces and their dependences are
of considerable interest, as one can see in  ~\cite{LDTS},
~\cite{TB}, ~\cite{CCP}, ~\cite{STLA}, ~\cite{STMA2} and
~\cite{MPDM}.

Many of these properties will be formulated in the Preliminaries
section. We would also like to attain a supertropical analog to the
classical eigenspace decomposition, (i.e., eigenvectors
corresponding to distinct eigenvalues are linearly independent, and
the generalized eigenvectors generate~$R^n$), but we encounter the
example of~\cite[Example~5.7]{STMA2} where the eigenvectors of
distinct eigenvalues are supertropically dependent, extensively
studied in Section~\ref{TPA}. Our objective in this paper is to
understand how such an example arises, and how it can be
circumvented, either by introducing the \textbf{difference
criterion} of Definition~\ref{difcrit} or by passing to generalized
eigenspaces in \S~\ref{geneig}.

\section{Preliminaries}\label{sec:pre}

In this section, we present well-known and recent results of
tropical polynomials. Then we introduce properties of matrices and
vectors in the tropical structure, with  definitions extended to the
supertropical framework.

\subsection{Tropical Polynomials}\label{sec:pre}

\begin{nota}\label{nuhat}
$ $

Throughout, for each element~$a\in R$, we choose an
element~$\hat{a}\in \mathcal{T}\text{ such that } \hat{a} \cong_\nu
a.$ (We define $0_R ^ \nu =  0_R ,$ so $ \widehat{0_R ^ \nu} =
0_R.$)

\noindent Likewise, for $\bfa = (a_1, \dots, a_n)$, $\bfa ^\nu$
denotes $(a_1^\nu, \dots, a_n^\nu)$ and  $\widehat{\bfa} $ denotes
$(\widehat{a_1}, \dots, \widehat{a_n})$. The same holds for matrices
and for polynomials (according to their coefficients).\end{nota}

\vskip 0.15 truecm


\begin{df} Let $k\in \mathbb{N}$. Defining $b=a^k$ to be the tropical product of $a$ by itself $k$ times (i.e., $a^{\varodot k}=a\varodot \cdots \varodot a=a+\cdots +a=ka$), we may consider that  $a$ is a~$k$-\textbf{root} of~$b$,  denoted as~$a= \sqrt[k]{b}$. This operation is well-defined on $\mathcal{T}$.
\end{df}

Clearly, any tropical polynomial  takes the value of the dominant
monomial along the~$\mathcal{T}$-axis. That having been said, it is
possible that some monomials in the polynomial would not dominate
for any~$x\in \mathcal{T}$.

\begin{df}
Let $f(x)=\sum_{i=0}^n \alpha_ix^{n-i}\in R[x]$ be a tropical polynomial. We
call monomials in $f(x)$ that dominate for some~$x\in R$
\textbf{essential}, and monomials in $f(x)$ that do not dominate for
any $x\in R$ \textbf{inessential}. We write $f^{es}(x)=   \sum_{k\in
I} \alpha_kx^{n-k}\in R[x],$ where $\alpha_kx^{n-k}$ is an essential monomial $
\forall k\in I$, called \textbf{the essential polynomial} of $f$.
 \end{df}

In the classical  sense, a  root of a tropical polynomial   can only be $0_R$, which occurs if and only if the polynomial has  constant
term $0_R$. We would like the roots to indicate the factorization
of the polynomial, which leads to the
following tropical definition of a root.

\begin{df}\label{root}  We define an element $r \in R$ to be a \textbf{root} of a
tropical polynomial $f(x)$ if~$f(r)\models_{gs} 0_R,$ i.e., $f(r)$
is a ghost.
\end{df}

 We refer to roots of a polynomial being obtained as a simultaneous value of two leading tangible monomials as \textbf{corner roots}, and to  roots that are being obtained from one leading  ghost monomial as \textbf{non-corner roots}.
 We factor polynomials viewing them as functions. Then, for every corner root $r$ of $f$, we may write $f$ as $(x+r)^k g(x)$  for some $g(x)\in R[x]$ and $k\in \mathbb{N}$, where $k$ is the difference between the exponents of the tangible essential monomials attaining $r$.
\vskip 0.15 truecm


\vskip 0.2 truecm

\subsection{Matrices}

\noindent As defined over a ring, for matrices~$A=(a_{i,j})\in M_{n\times m}(R),\
B=(b_{i,j})\in M_{s\times t}(R)$
$$\begin{cases}
A+B=(c_{i,j}) :\ c_{i,j}=a_{i,j}+b_{i,j},&  \text{defined iff }n=s,m=t\ ,\\
\\
\ \ AB=(d_{i,j})\ :\ d_{i,j}=\sum_{k\in[n]}a_{i,k}b_{k,j},&
\text{defined iff }\ m=s\ .
\end{cases}$$

\begin{df}\label{pt} Let $\pi\in S_n$ and $A=(a_{i,j})\in M_n(R)$. The \textbf{permutation} $\mathbf{\pi}$ \textbf{of} $\mathbf{A}$ is the word $$a_{1,\pi(1)}a_{2,\pi(2)}\cdots a_{n,\pi(n)}.$$ The word $a_{1,1}a_{2,2}\cdots a_{n,n}$ is denoted as the \textbf{identity or Id-permutation}, corresponding to the diagonal of $A$.
 We write  a permutation  of $A$ as a product of disjoint cycles~$C_1,\dots,C_t$, where~$\{C_i\}$ corresponds to the disjoint cycles composing $\pi$.

We define the tropical \textbf{trace} and   \textbf{determinant} of
$A$  to be $$\tr(A)=\sum_{k\in[n]} a_{k,k}\ \ \ \text{and}\ \ \
\det(A)=\sum_{\sigma \in S_n}a_{1,\sigma(1)}\cdots
a_{n,\sigma(n)},$$ respectively.

\noindent In the special case where $A\in M_n(R)$, we refer to any
entry attaining the trace as \textbf{a dominant diagonal entry}. We
call $ a_{1,\sigma(1)}\cdots a_{n,\sigma(n)}$ the \textbf{weight}
contributed by $\sigma$ to the determinant, and any permutation
whose weight has the same $\nu$-value as the determinant is
\textbf{a dominant permutation of $\mathbf{A}$}.
\end{df}

If there is a single dominant permutation, its weight equals the
determinant. \vskip 0.1 truecm

Unlike over a field, the tropical concepts of singularity, invertability and factorizability do not coincide.
We would like the determinant to indicate the singularity of a matrix. Hence, we define a matrix~
$A\in \M_n(R)$  to be  \textbf{tropically singular} if there exist
at least two different dominant permutations. Otherwise the matrix
is \textbf{tropically nonsingular}. Consequently,  a matrix~$A\in
\M_n(R)$ is supertropically singular if~$\det(A)\models_{gs} 0_R$
and supertropically nonsingular if~$\det(A) \in \mathcal T$. A
matrix~$A$ is \textbf{strictly singular} if~$\det(A)=0_R$.

 A surprising result in this context is that the product of two nonsingular matrices might be singular, but we do have:

\begin{thm} \label{det}
For $n\times n$ matrices $A,B$ over the supertropical  semiring $R$,
we have $$\det(AB) \models_{gs} \det (A)\det (B).$$
\end{thm}

This theorem has been proved  in ~\cite[ Theorem 3.5]{STMA} due to considerations of graph theory, but also in ~\cite[Proposition 2.1.7]{G.Ths} by using the  transfer principles (see ~\cite[Theorem 3.3 and Theorem
3.4 ]{LDTS}). These theorems allow one to obtain such results automatically in a wider class of semirings, including the supertropical semiring.

\vskip 0.3 truecm

\begin{df}
Suppose $\mathcal{R}$ is a semiring. An \textbf{$ \mathcal{R}$-module} $V$ is a semigroup $(V,+, 0_V)$ together with
scalar multiplication $\mathcal{R} \times V \rightarrow V$ satisfying the following properties for all $r_i \in \mathcal{R}$ and $v,w \in V$:
\begin{enumerate}
\item $r(v + w) = rv + rw$

\item $(r_1 + r_2)v = r_1v + r_2v$

\item $(r_1r_2)v = r_1(r_2v)$

\item $1_{\mathcal{R}}v = v$

\item $r\cdot 0_V = 0_V$

\item $0_{\mathcal{R}}\cdot v = 0_V$.
\end{enumerate}
\end{df}

For any semiring $R$, let   $R^n$ be the free module of rank $n$
over $R$. We define the \textbf{standard base}  to be
$e_1,\dots,e_n$, where

$$e_i=
\begin{cases}
1_{\mathcal{T}}=1_{R},& \text{in the}\ i^{th}\ \text{coordinate}\\
0_{\mathcal{T}}=0_{R},& \text{otherwise}
\end{cases}.$$
\vskip 0.3 truecm

 The tropical \textbf{identity matrix} is  the~$n\times n$ matrix with the standard base for its columns. We denote this matrix as~$I_{\mathcal{T}}=I.$

 A matrix $A\in \M_n({R})$  is \textbf{invertible}  if there  exists  a matrix~$B\in \M_n({R})$ such that~$AB=BA=I.$

From now on $\mathcal F := \mathcal T \cup \mathcal G \cup \{
0_{\mathcal F} \}$, where its set $\mathcal T $ is presumed to be a
group, and $\mathcal G$ is its ghost elements. We write $V =
\mathcal F ^ n,$ with the standard base $\{e _1, \dots, e_n\}$.

\begin{df} We define  vectors $v_1,\dots,v_k$ in  $V$ to be \textbf{(supertropically) dependent} if there exist~$a_1,\dots,a_k\in \mathcal T$ such that $\sum_{i\in[k]} a_i v_i\models_{gs} \overrightarrow{0_{\mathcal F}}$. Otherwise, this set of tropical vectors is called
\textbf{independent}.

We say that subspaces  
 $V_1,...,V_k$ of ${\mathcal F}^n$, are \textbf{(supertropically)
dependent}, if there are tangible~$v_i \in V_i$ which are (supertropically)
dependent.
\end{df}

By~\cite[Theorem~6.5]{STMA}, vectors $v_1,...,v_n\in {\mathcal F}^n$ are dependent iff~$\det(V)\in\mathcal{G}\cup\{0_{\mathcal{F}}\}$, where $V$ is the matrix having ~$v_1,...,v_n$ for its columns.

We define two types of special matrices:

\begin{df} An $n\times n$ matrix $P=(p_{i,j})$ is a \textbf{permutation matrix} if there exists $\pi\in S_n$ such that $$p_{i,j}=
\begin{cases}
0_{\mathcal F}, & j\ne \pi(i)\\
 1_{\mathcal F},&  j=\pi(i)
\end{cases}.$$  Since $\forall\pi\in S_n\ \exists ! \sigma\in S_n:\ \sigma=\pi^{-1}$ and $1_{\mathcal F}$ is invertible, a permutation matrix is always invertible.

\noindent An $n\times n$  matrix~$D=(d_{i,j})$ is  a
\textbf{diagonal matrix} if $$\exists\ a_1,\dots,a_n\in {\mathcal
F}:\ d_{i,j}=
\begin{cases}
0_{\mathcal F},&  j\ne i\\
 a_{i}, &  j=i
\end{cases},$$ which is invertible if and only if $\det(D)$ is invertible (i.e., $a_i\in \mathcal T, \ \forall i$).
\end{df}

\begin{rem}(See ~\cite[Proposition 3.9]{STMA})\label{inv}
 A tropical matrix $A$ is invertible if and only if it is  a product of a permutation matrix  and an invertible diagonal matrix.  These types of products are called \textbf{generalized permutation matrices}, that is $(d_{i,j})$ such that $$\exists\ a_1,\dots,a_n\in \mathcal T,\ \pi\in S_n:\ d_{i,j}=
\begin{cases}
0_{\mathcal F},&  j\ne \pi(i)\\
 a_{i},&  j=\pi(i)
\end{cases}.$$
\end{rem}

We define  three types of tropical elementary matrices, corresponding to the three elementary  matrix operations,  obtained by applying one such operation to the identity matrix.
\vskip 0.2 truecm

\textbf{A transposition matrix}  is obtained from the  identity
matrix by switching two rows (resp. columns). This matrix is
invertible: $E_{i,j}^{-1}=E_{i,j},$ and a product of transposition matrices yields  a permutation matrix.

\vskip 0.2 truecm

\textbf{An elementary diagonal multiplier}  is obtained from the
identity matrix where one row (resp. column) has been multiplied by
an invertible scalar.  This matrix is invertible:
$E^{-1}_{\alpha\cdot i^{th}row}=E_{\alpha^{-1}\cdot i^{th}row},$ and
a product of diagonal multipliers    yields  an invertible diagonal
matrix.

\vskip 0.2 truecm

\textbf{A Gaussian matrix} is defined to differ from the identity matrix by having a non-zero entry in a non-diagonal position. We denote as $E_{i^{th}row+\alpha\cdot  j^{th}row}$ the elementary Gaussian matrix  adding row $j$, multiplied  by $\alpha$, to row $i$. By Remark \ref{inv}, this matrix is not invertible.

\begin{df} A nonsingular matrix $A=(a_{i,j})$ is defined as \textbf{definite}
if $$\det(A)=0=a_{i,i},\ \forall i.$$\end{df}







\subsubsection{The supertropical approach}

Having established that algebraically $\mathcal{G}\cup\{-\infty \}$
and~$\models_{gs}$ effectively take the role of singularity and
equality over ${\mathcal F}$, we would like to extend additional
definitions to the supertropical setting, using ghosts for zero.

\noindent A \textbf{quasi-zero} matrix $Z_\mathcal{G}$ is a matrix
equal to $0_{\mathcal F}$ on the diagonal, and whose off-diagonal
entries are ghost or $0_{\mathcal F}$.

\noindent A \textbf{diagonally dominant} matrix is a nonsingular matrix
with a dominant permutation along the diagonal.

\noindent A \textbf{quasi diagonally dominant} matrix $D_\mathcal{G}$ is a
diagonally dominant matrix  $A$ whose off-diagonal entries are ghost or~$0_{\mathcal
F}$.

 \noindent A \textbf{quasi-identity} matrix $I_\mathcal{G}$ is a nonsingular, multiplicatively idempotent matrix
equal to $I + Z_\mathcal{G}$, where $Z_\mathcal{G}$ is a quasi-zero matrix.

Thus, every quasi-identity matrix $I_\mathcal{G}$ is quasi diagonally dominant.
Using the tropical determinant, we attain the tropical analog for
the well-known~\textit{adjoint}.
\begin{df}\label{adj} The ${r,c}$-\textbf{minor} $A_{r,c}$ of a matrix $A = (a_{i,j})$ is obtained by deleting  row~$r$ and column $c$ of $A$.
 The \textbf{adjoint matrix} $\adj(A)$ of $A$ is defined as the matrix~$(a'_{i,j} )$, where~$a'_{i,j} =\det(A_{j,i})$. When $\det(A)$ is invertible,
 the matrix $A^{\nabla}$ denotes $$\frac 1 {\det(A)}\adj(A).$$\end{df}
\vskip 0.2 truecm

 Notice that $\det(A_{j,i})$ may be obtained as the sum of all permutations in $A$  passing through $a_{j,i}$, but with $a_{j,i}$ deleted:    $$\det(A_{j,i})=\sum_{\tiny{\begin{array}{cc}\sigma\in S_n:\\\sigma(j)=i\end{array}}}a_{1,\sigma(1)}\cdots a_{j-1,\sigma(j-1)}a_{j+1,\sigma(j+1)}\cdots a_{n,\sigma(n)}.$$ When writing  each permutation as the product of  disjoint cycles, $\det(A_{j,i})$ can be presented as: $$\det(A_{j,i})=\sum_{\tiny{\begin{array}{cc}\sigma\in S_n:\\\sigma(j)=i\end{array}}}(a_{i,\sigma(i)}a_{\sigma(i),\sigma^2(i)}\cdots a_{\sigma^{-1}(j),j})C_{\sigma},$$ where~$C_{\sigma}$ is the product of
the remaining cycles.

\begin{df}  We say that $A^{\nabla}$ is the \textbf{quasi-inverse} of $A$ over ${{\mathcal F}}$, denoting $$I_A = AA^{\nabla}\text{ and }I'_A=A^{\nabla}A,$$ where $I_A,I'_A$ are quasi-identities  (see ~\cite[Theorem 2.8]{STMA2}).\end{df}

These supertropical definitions provide a tropical version for two
well-known algebraic properties, proved in Proposition 4.8.~and
Theorem 4.9.~of ~\cite{STMA}.

\begin{pro}\label{adjgs} $ \adj(AB) \models_{gs} \adj(B) \adj(A)$.\end{pro}

As a result, one concludes from the fourth property of~$\models_{gs}$ (see Definition~\ref{gs}) and Theorem~\ref{det} that~$(AB)^\nabla\models_{gs} B^\nabla A^\nabla$, when~$AB$ is nonsingular.

\begin{thm}\label{detadj} $ $
\begin{enumerate}[\ \ \ (i)]
\item $\det(A\cdot \adj(A))= \det(A)^n$ .

\item $\det(\adj(A))= \det(A)^{n-1}$.
\end{enumerate}
\end{thm}

\begin{rem} (see ~\cite[Remark 2.18]{PI&CP}) \label{adjnf}  For a
definite matrix $A$ we have
$$A^{\nabla}=\frac{1}{\det(A)}\adj(A)=\adj(A), $$ which is also
definite. \end{rem}

The following lemma has been proved in ~\cite[Lemma 3.2]{PI&CP}, and
states the connection between multiplicity of the determinant  and
the quasi-inverse matrix:

\begin{lem}\label{invnab} Let $P$ be an invertible matrix and $A$ be nonsingular.

\begin{enumerate}[\ \ \ (i)]

\item $P^\nabla =P^{-1}$.

\vskip 0.1 truecm

\item $\det(PA)=\det(P)\det(A)$.

\vskip 0.1 truecm

\item $(PA)^\nabla =A^\nabla P^\nabla$.

\vskip 0.1 truecm

\item If~$A=P\bar{A}$, where~$\bar{A}$ is the definite form of~$A$ with   left normalizer~$P$, then~$A^\nabla =\bar{A}^{\nabla}P^{-1}$ where~$\bar{A}^{\nabla}$
is  definite,  with  right normalizer~$P^{-1}$.
\end{enumerate}
\end{lem}

\vskip 0.35 truecm

\begin{center}\textbf{Matrix invariants}\end{center}

Let $A \in \M_n({{\mathcal F}})$. We continue the supertropical
approach by defining $ v \in V$, not all singular, such that
$\exists \lambda \in \mathcal T\cup\{0_{\mathcal F}\}\text{ where
}Av\models_{gs} \lambda v,$ to be a \textbf{supertropical
eigenvector}
 of $A$ with a \textbf{supertropical eigenvalue}~$\lambda$, having an
\textbf{eigenmatrix} $A+\lambda I$. The \textbf{eigenspace} $V_\lambda$ is the set of
eigenvectors with eigenvalue $\lambda$.

\noindent The \textbf{characteristic polynomial} of $A$ (also called
the maxpolynomial, cf.\cite{MA}) is defined to be
$$f_A(x)=\det(xI+A).$$ The tangible value of its roots  are the
eigenvalues of $A$, as shown in~\cite[Theorem~7.10]{STMA}. Following
to Definition \ref{root}, we may have \textit{corner eigenvalues}
and \textit{non-corner eigenvalues}.

 The coefficient of $x^{n-k}$ in this polynomial is the sum of  determinants of all $k\times k$ \textbf{principal sub-matrices}, otherwise known as the trace of the $k^{th}$ compound matrix of~$A$.
 Thus, this coefficient, which we denote as $\alpha_k$,  takes the dominant value among the permutations on all subsets of indices of size $k$: $$\alpha_k=\sum_{\tiny{\begin{array}{c}I\subseteq[n]:\\|I|=k\ \end{array}}}\sum_{\sigma\in S_k}\prod_{i\in I}a_{i,\sigma(i)}.$$
 When $\alpha_k \in \mathcal T,$ we define the \textbf{index set of} ${\alpha_k}$,  denoted by $\Ind_k$, a set $I\subseteq[n]$ on which the dominant permutation defining $\alpha_k$ is obtained.

Let $f_A(x)=\sum_{i=0}^n \alpha_ix^{n-i} $ be the characteristic polynomial  of $A$, with the essential  polynomial $$f^{es}_A(x)=\sum_k \alpha_{i_{_k}}x^{n-{i_{_k}}}.$$
Let $\lambda$ be the corner eigenvalue obtained between the essential monomial  $\alpha_{i_{_{k-1}}}x^{n-(i_{_{k-1}})}$ and the subsequent essential monomial  $\alpha_{i_{_k}}x^{n-i_{_k}}$. We denote
 $I_\lambda=\Ind_{i_k}\setminus\Ind_{i_{k-1}}.$

\begin{thm}\label{alg} (The eigenvectors algorithm, see ~\cite[Remark 5.3 and Theorem 5.6]{STMA2}.)
Let $t\in I_\lambda.$  The tangible value of
the $t^{th}$-column  of~$\adj(\lambda I+A)$ (see Notation~\ref{nuhat}), is a tropical eigenvector
of $A$ with respect to the eigenvalue~$\lambda$.
\end{thm}\noindent This algorithm will be demonstrated in \S 3.2.

\vskip 0.25 truecm The Supertropical Cayley-Hamilton Theorem has
been proved in ~\cite[Theorem 5.2]{STMA}, and is as follows:

\begin{thm}\label{H-C} Any matrix $A$ satisfies its tangible characteristic polynomial $f_A$,  in the sense that~$f_A(A)$ is ghost. \end{thm}
\noindent  One can find a  combinatorial proof in ~\cite{CH} and  a proof using the transfer principle in ~\cite{LDTS}.

In analogy to the classical theory, we have
\begin{pro}\label{pev} (\cite[Proposition~7.7]{STMA}) The roots of the polynomial $f_A(x)$
are precisely the supertropical eigenvalues of $A$.\end{pro}

\begin{rem}\label{eigen1}
Recall that a supertropical polynomial is $r$-\textbf{primary} if it has the unique supertropical root $r$. It is well-known that  any tropical $r$-primary polynomial has the form $(x + r)^m$  for some $m\in \mathbb{N}$, and  any tropical essential polynomial $f_A$ can be factored as a function to a product  of primary polynomials, and thus of the form $\prod _i g_i$ where  $g_i = (x + r_i)^{m_i}.$
The supertropical version of this is given in \cite[Theorem~8.25 and~Theorem~8.35]{STA}.
\end{rem}

 Another classical property attained in this extended structure is:

\begin{pro}\label{pev}
If $ \lambda \in \mathcal T\cup\{0_{\mathcal F}\}$ is a supertropical
eigenvalue of a matrix $A \in \M_n({{\mathcal F}})$ with eigenvector
$v$, then $\lambda^i$ is a supertropical eigenvalue of  $A^i$, for
every $i\in \mathbb{N}$, with respect to the same
eigenvector.\end{pro}

\begin{thm}\label{CPIP} Let $A$ be a nonsingular matrix.

\begin{enumerate}
\item  (~\cite[Theorem~3.6]{CP}) For any $m\in \mathbb{N}$ we have $$f_{A^m}(x^m) \models_{gs} (f_A(x))^m,$$ implying that the~$m^{th}$-root of every corner eigenvalue of~$A^m$ is a corner eigenvalue of~$A$.
\vskip 0.25 truecm

\item (~\cite[Theorem~4.1]{EP}) For $A^\nabla$,  the quasi-inverse of $A$, we have $$\det(A)f_{A^\nabla}(x) \models_{gs} x^nf_A(x^{-1}),$$ implying that the inverse   of every corner eigenvalue of ~$A^\nabla$ is a corner eigenvalue of~$A$.
\end{enumerate}
\end{thm}

\section{Dependence of eigenvectors}

A well-known decomposition of $F^n$, where $F$ is a field, is the
decomposition to  eigenspaces of a matrix $A\in \M_n(F)$. In
particular, this decomposition is obtained when the eigenvalues are
distinct since, in the classical case,  eigenspaces of distinct
eigenvalues are linearly independent, which compose a basis for
$F^n$. In the tropical case,  considering that dependence  occurs
when a tropical linear combination
ghost-surpasses~$\overrightarrow{0_{\mathcal F}}$, such a property
need not necessarily hold.

In the upcoming section we analyze the dependence between
eigenvectors, using their definition according to the algorithm
described in Theorem \ref{alg}. We present special cases in which this undesired
dependence is resolved.

\begin{df}\label{difcrit} The matrix $A$ satisfies the \textbf{difference criterion} if the
sets $I_\lambda,$ such that $\lambda$ is a corner root of $f_A$, are
disjoint.
\end{df}

\subsection{Eigenspaces in lower dimensions}

the In the following proposition, we verify   independence of
eigenvectors having distinct eigenvalues, for dimensions  $n  =
2,3$.

\begin{pro}\label{es2-3} Let $A=(a_{i,j})$ be a nonsingular  $n\times n$ matrix, where $n\in\{2,3\}$, with a tangible characteristic polynomial (coefficient-wise) and $n$ distinct eigenvalues. Then the eigenvectors of $A$ are tropically independent. \end{pro}

\begin{proof}

$ $

\noindent \underline{The $2\times 2$ case:}

\noindent Let $f_A(x)=x^2+\tr(A)x+\det(A)$ be the characteristic
polynomial of $A$.  If $A$ has two distinct eigenvalues, then these
must be $\lambda_1=\tr(A)$ and $\lambda_2=\frac{\det(A)}{\tr(A)}$.

We must have $\lambda_1>\lambda_2,$  for otherwise either
$$f_A(\lambda_2)=\frac{\det(A)}{\tr(A)}\left(\frac{\det(A)}{\tr(A)}+\tr(A)^\nu\right)=\left(\frac{\det(A)}{\tr(A)}\right)^2\in
\mathcal T,$$ or $\lambda_1=\lambda_2$, which means the polynomial
has one root with multiplicity $2$.

\noindent Without loss of generality, we may assume that $\tr(A)=a_{1,1}$. According to the algorithm, since $I_{\lambda_1}=\{1\}$, $\lambda_1$ has the eigenvector obtained by the tangible value of the \textit{first} column of its eigenmatrix. Since $I_{\lambda_2}=\{2\}$, $\lambda_2$ has the eigenvector obtained by the tangible value of the \textit{second} column of its eigenmatrix.

\noindent The determinant is either:
$$\det(A)=a_{1,1}a_{2,2},\text{ where }a_{1,1}>a_{2,2} \text{ and } a_{1,1}a_{2,2}>a_{1,2}a_{2,1},$$
(and then  the eigenvalues are~$a_{1,1}$ and~$a_{2,2},$) or
$$\det(A)=a_{1,2}a_{2,1},\text{ where }a_{1,1}a_{2,2}<a_{1,2}a_{2,1},$$ (and then  the eigenvalues are $a_{1,1}$ and $\frac{a_{1,2}a_{2,1}}{a_{1,1}},$ satisfying $a_{1,1}>\frac{a_{1,2}a_{2,1}}{a_{1,1}}>a_{2,2}$).

\noindent In both cases,  the first column of $\adj(A+\lambda_1I)$
is $(a_{1,1}, a_{2,1})$ and the second column
of~$\adj(A+\lambda_2I)$ is $(a_{1,2}, a_{1,1})$, which are
tropically independent since $a_{1,1}^2>a_{1,2}a_{2,1}$.

\vskip 0.2 truecm

\noindent \underline{The $3\times 3$ case:}

This case indicates key techniques for understanding and motivating
the general proof on matrices satisfying the difference criterion  in \S 3.3.1.

\noindent Let~$f_A(x)=x^3+\tr(A)x^2+\alpha x+\det(A)$ be the
characteristic polynomial of~$A$, recalling that $\alpha$ is the sum
of the determinants of all of the principle $2\times 2$
sub-matrices. We assign~$\tr(A)$ to be~$a_{1,1}$, i.e.,
\begin{equation}\label{3-1} a_{1,1}>a_{t,t}\ \forall t\ne
1.\end{equation}

\noindent For the determinant we have six permutations of
$S_3$. In order to obtain three distinct eigenvalues, we must have
\begin{equation}\label{3-2}
\lambda_1=\tr(A)>\lambda_2=\frac{\alpha}{\tr(A)}>\lambda_3=\frac{\det(A)}{\alpha},\end{equation}
for otherwise $\exists t,s:\ f_A(\lambda_t)\in \mathcal T$ or
$\lambda_t=\lambda_s$. Thus  \begin{equation}\label{3-3}
\lambda_1\lambda_2=\alpha\ \ \text{and}\ \
\lambda_1\lambda_2\lambda_3=\det(A).\end{equation} As a result,
$\Ind_1\subseteq\Ind_2$; otherwise, $a_{1,1}$ together with $\alpha$
yields a permutation whose weight is dominated by $\det(A)$, and we
get $\lambda_1=a_{1,1}<\frac{\det(A)\cdot a_{1,1}}{\alpha\cdot
a_{1,1}}=\lambda_3,$  contrary to~\eqref{3-2}. \vskip 0.2 truecm

Therefore,
$$\begin{cases}\,I_{\lambda_1}=\{1\}\setminus\emptyset=\{1\}\\
I_{\lambda_2}=\{1,j\}\setminus\{1\}=\{j\},\\
I_{\lambda_3}=\{1,j,k\}\setminus\{1,j\}=\{k\},\end{cases}$$where
$1,j,k$ are distinct. Without loss of generality, we may take $j=2$
and~$k=3$, and obtain the  eigenmatrices:

$\ \ \ \ A+\lambda_1I=\left(
\begin{array}{ccc}
\lambda_1 & a_{1,2} & a_{1,3}\\
a_{2,1}  & \lambda_1 & a_{2,3}\\
a_{3,1}  & a_{3,2}& \lambda_1
\end{array}
\right),\ \text{since $\tr(A)=a_{1,1}>a_{t,t},\ \ \forall t\ne 1$,
by~\eqref{3-1}},$ \vskip 0.25 truecm

$\ \ \ \ A+\lambda_2I=\left(
\begin{array}{ccc}
\lambda_1 & a_{1,2} & a_{1,3}\\
a_{2,1}  & \lambda_2 & a_{2,3}\\
a_{3,1}  & a_{3,2}& \lambda_2
\end{array}
\right),\ \text{since
$\underbrace{\overbrace{\frac{\alpha}{\tr(A)\cdot a_{t,t}}}}_{\geq
1}a_{t,t}\geq a_{t,t}$},$

\noindent because~$\tr(A)\cdot a_{t,t}$ is a summand of~$\alpha$,
$\forall t\ne 1$, and \vskip 0.25 truecm

$\ \ \ \ A+\lambda_3I=\left(
\begin{array}{ccc}
\lambda_1 & a_{1,2} & a_{1,3}\\
a_{2,1}  & \beta & a_{2,3}\\
a_{3,1}  & a_{3,2}& \lambda_3
\end{array}
\right),\ \text{since $\underbrace{\overbrace{\frac{\det(A) }
{\alpha\cdot a_{3,3}}}}_{\geq 1}a_{3,3}\geq a_{3,3}$},$

\noindent where $\beta=\max\{a_{2,2},\lambda_3\}$,
since~$\alpha\cdot a_{3,3}$ is a summand in~$\det(A)$. \vskip 0.2
truecm

Recalling the algorithm in Theorem \ref{alg}, we let $W$ be the
matrix with the (tangible value of the) eigenvectors for its columns

$$W=\left(
\begin{array}{ccc}
\lambda_1^2                                           & a_{1,2}\lambda_2+a_{1,3}a_{3,2}                    & a_{1,3}\beta+a_{1,2}a_{2,3}\\
 &&\\
a_{2,1}\lambda_1+a_{2,3}a_{3,1}          & \lambda_1\lambda_2                                          & a_{2,3}\lambda_1+a_{2,1}a_{1,3}\\
&&\\
a_{3,1}\lambda_1+a_{3,2}a_{2,1}          & a_{3,2}\lambda_1+a_{3,1}a_{1,2}                     & \lambda_1\beta+a_{1,2}a_{2,1}
\end{array}
\right).$$
We get~$W_{3,3}=\lambda_1\lambda_2$, since

$\text{if }\alpha=\lambda_1\lambda_2=a_{1,2}a_{2,1}\text{ then }\lambda_1a_{2,2}< \alpha,\ \lambda_1\lambda_3< \alpha\Rightarrow\lambda_1\beta+a_{1,2}a_{2,1} =\lambda_1\lambda_2,\text{ and}$

$\text{if }\alpha=\lambda_1\lambda_2=a_{1,1}a_{2,2}\text{ then }\lambda_3<\frac{a_{2,2}}{a_{1,1}}a_{1,1}=\lambda_2\Rightarrow\beta=a_{2,2},\ \lambda_1a_{2,2}+a_{1,2}a_{2,1}=\lambda_1\lambda_2.$

\noindent Due to  relations \eqref{3-1}-\eqref{3-3}, all non-identity permutations in $\det(W)$:

$\begin{cases}
\lambda_1^4a_{2,3}a_{3,2}=\lambda_1^2\lambda_1a_{1,1}a_{2,3}a_{3,2}\leq\lambda_1^2\lambda_1\lambda_1\lambda_2\lambda_3 \ \ \ \text{ and}& \lambda_1^2\lambda_2\lambda_3a_{1,3}a_{3,1}<\lambda_1^2\lambda_1\lambda_2a_{1,3}a_{3,1},\\

\lambda_1\lambda_m\lambda_ra_{i,j}a_{j,l}a_{l,i}\leq \lambda_1^2\lambda_2(\lambda_1\lambda_2\lambda_3)&,\ i,j,l\text{ distinct},\\

\lambda_m\lambda_r(a_{i,j}a_{j,i})(a_{i,l}a_{l,i})\leq\lambda_1^2(a_{i,j}a_{j,i})(a_{i,l}a_{l,i})&,\ i,j,l\text{ distinct},\\

\lambda_m(a_{i,j}a_{j,i})(a_{i,l}a_{l,t}a_{t,i})\leq\lambda_1(\lambda_1\lambda_2)(\lambda_1\lambda_2\lambda_3)&,\ j\ne i,t,l\text{ distinct},\\

(a_{i,j}a_{j,i})(a_{k,l}a_{l,k})(a_{t,s}a_{s,t})\leq(\lambda_1\lambda_2)^3&,\ i\ne j, k\ne l, t\ne s,\\

\text{and }\ \ (a_{i,j}a_{j,k}a_{k,i})(a_{l,t}a_{t,s}a_{s,l})\leq(\lambda_1\lambda_2\lambda_3)^2&,\ i,j,k\text{ distinct},\ s,t,l\text{ distinct},\end{cases}$

\noindent are strictly dominated by $\lambda_1^2(\lambda_1\lambda_2)(\lambda_1\lambda_2)$.

\end{proof}

We further study this property in the generalization proved in
Theorem \ref{es-n}. The cases  in Step~3 of its proof are
demonstrated above.

\subsection{The pathology  appears}\label{TPA}

We follow Example \ref{exa}, introduced in ~\cite{STMA2}, to show
how independence of eigenspaces might fail for dimensions higher
then $3$,  due to the increased variety of indices. While applying
the eigenvectors-algorithm, we utilize a supertropical analog of
classical Gaussian elimination, treating the ghosts as
``zero-elements''. This illustrative example will provide the
motivation for  Theorem \ref{es-n}, Conjecture \ref{esnab1} and
Conjecture \ref{esnab2}, generalizing the connection of the index
sets   to the dependence of the eigenvectors.

\begin{exa}\label{exa}

Let $$A=
\left(
\begin{array}{cccc}
10 & 10 & 9 & - \\
9  & 1 & - & - \\
-  & - & - & 9\\
9 & - & - & -
\end{array}
\right).$$

\noindent The characteristic polynomial of $A$ is
$$f_A(x)=x^4+10x^3+19x^2+27x+28,$$ obtained from the permutations $(1),\ (1\
2),\ (1\ 3\ 4),\ (1\ 3\ 4)(2),$ respectively. Therefore,
\begin{equation}\label{ind}\begin{cases}I_{\lambda_1}=\{1\}\setminus\emptyset=\{1\},\\ I_{\lambda_2}=\{1,2\}\setminus\{1\}=\{2\},\\ I_{\lambda_3}=\{1,3,4\}\setminus\{1,2\}=\{3,4\},\\ I_{\lambda_4}=\{1,2,3,4\}\setminus\{1,3,4\}=\{2\}\end{cases}\end{equation} where $\lambda_1=10,\ \lambda_2=9,\ \lambda_3=8\
\text{and}\ \lambda_4=1,$ are the eigenvalues of $A$. As we saw in \S 3.1, the overlap of the
second and fourth sets cannot occur in lower dimensions.

\end{exa}

\noindent The eigenmatrices and eigenvectors are as follows:

\medskip
 \noindent \underline{For $\lambda_1:$}
$$A+10I=
\left(
\begin{array}{cccc}
10^\nu & 10 & 9 & - \\
9  & 10 & - & - \\
-  & - & 10 & 9\\
9 & - & - & 10
\end{array}
\right),$$ and the tangible value of the first column of its adjoint
is $$v_1=\left(30,29,28,29\right)=28\left(2,1,0,1\right).$$ This can
also be obtained when multiplying the eigenmatrix by $$E_{4^{th}\
row+1\cdot 3^{rd}\ row}^2E_{4^{th}\ row+1\cdot 2^{nd}\
row}E_{2^{nd}\ row+1^{st}\ row}E_{1,4}$$ on the left:
$$\left(
\begin{array}{cccc}
9 & - & - & 10 \\
9^\nu  & 10 & - & 10 \\
-  & - & 10 & 9\\
10^\nu & 10^\nu & 12^\nu & 11^\nu
\end{array}
\right),$$ and solving the tropically linear system
$$\begin{cases}
9x+10w\in \mathcal{G},\\
 10y+10w\in \mathcal{G},\\
 10z+9w\in \mathcal{G},
\end{cases}$$ which yields $\left(11,10,9,10\right)=9\left(2,1,0,1\right),$ a multiple of $v_1$.

\medskip

\noindent \underline{For $\lambda_2:$}
$$A+9I=
\left(
\begin{array}{cccc}
10 & 10 & 9 & - \\
9  & 9 & - & - \\
-  & - & 9 & 9\\
9 & - & - & 9
\end{array}
\right),$$ and the tangible value of the second column of its adjoint
is $$v_2=\left(28,28,28,28\right)=28\left(0,0,0,0\right).$$ This can
also be obtained when multiplying the eigenmatrix by $$E_{4^{th}\
row+2\cdot 3^{rd}\ row} E_{4^{th}\ row+1\cdot 2^{nd}\ row}E_{2^{nd}\
row+1^{st}\ row}E_{1,4}$$ on the left:
$$\left(
\begin{array}{cccc}
9 & - & - & 9 \\
9^\nu  & 9 & - & 9 \\
-  & - & 10 & 9\\
10^\nu & 10^\nu & 9^\nu & 9^\nu
\end{array}
\right),$$ and solving the tropically linear system
$$\begin{cases}
9x+9w\in \mathcal{G},\\
 9y+9w\in \mathcal{G},\\
 9z+9w\in \mathcal{G},
\end{cases}$$ which yields $\left(0,0,0,0\right),$ a multiple of $v_2$.

\medskip

\noindent \underline{For $\lambda_3:$}
$$A+8I=
\left(
\begin{array}{cccc}
10 & 10 & 9 & - \\
9  & 8 & - & - \\
-  & - & 8 & 9\\
9 & - & - & 8
\end{array}
\right),$$ and the tangible value of the third column of its adjoint
is $$v_3=\left(25,26,27,26\right)=25\left(0,1,2,1\right).$$ This can
also be obtained when multiplying the eigenmatrix by $$E_{4^{th}\
row+1\cdot 3^{rd}\ row}E_{4^{th}\ row+2\cdot 2^{nd}\ row}E_{2^{nd}\
row+1^{st}\ row}E_{1,4}$$ on the left:
$$\left(
\begin{array}{cccc}
9 & - & - & 8 \\
9^\nu  & 8 & - & 8 \\
-  & - & 8 & 9\\
11^\nu & 10^\nu & 9^\nu & 10^\nu
\end{array}
\right),$$ and solving the tropically linear system
$$\begin{cases}
9x+8w\in \mathcal{G},\\
 8y+8w\in \mathcal{G},\\
 8z+9w\in \mathcal{G},
\end{cases}$$ which yields $\left(7,8,9,8\right)=7\left(0,1,2,1\right),$ a multiple of $v_3$.

\medskip

\noindent  \underline{For $\lambda_4$}
$$A+1I=
\left(
\begin{array}{cccc}
10 & 10 & 9 & - \\
9  & 1^\nu & - & - \\
-  & - & 1 & 9\\
9 & - & - & 1
\end{array}
\right),$$ and the tangible value of the second column of its adjoint
is $$v_4=\left(12,27,28,20\right)=12\left(0,15,16,8\right).$$ This can
also be obtained when multiplying the eigenmatrix by
$$E_{4^{th}row+(-1)\cdot 1^{st}\ row}E_{4^{th}row+\cdot 2^{nd}\
row}E_{2^{nd}+(-1)\cdot 1^{st}\ row}$$ on the left:
$$\left(
\begin{array}{cccc}
10 & 10 & 9 & - \\
9^\nu  & 9 & 8 & - \\
-  & - & 1 & 9\\
9^\nu & 9^\nu & 8^\nu & 1^\nu
\end{array}
\right),$$ and solving the tropically linear system
$$\begin{cases}
10x+10y+9z\in \mathcal{G},\\
 9y+8z\in \mathcal{G},\\
 1z+9w\in \mathcal{G},
\end{cases}$$ which  yields~$\left(x,8,9,1\right),$ where~$x\leq 8$.

\medskip

\noindent From the fourth position of~$Av\models_{gs} \lambda v,$ we
get~$9x\models_{gs} 2$ which implies~$x=-7$. Thus the eigenvector
is~$(-7,8,9,1)=-7(0,15,16,8),$ a multiple of~$v_4$.

Next, we examine the dependence of  the eigenvectors, using the
matrix $W$ having these vectors for its columns:
$$W=
\left(
\begin{array}{cccc}
30 & 28 & 25 & 12 \\
29  & 28 & 26 & 27 \\
28  & 28 & 27 & 28\\
29 & 28 & 26 & 20
\end{array}
\right).$$ The determinant of~$W$ is $112^\nu$ and is obtained by
the permutations~$(1)(2)(3\ 4)$ and $(1)(2\ 4)(3).$ One can  see
that the ghost part of the product is  attained in the principal
sub-matrix $\{2,3,4\}\times\{2,3,4\}$, where the pathology of the
index sets occurs. We rewrite~$W$ using   the eigenvalues and the
entries of $A=(a_{i,j})$, in order to  understand this dependence:
$$W=
\left(
\begin{array}{ccccccc}
\lambda_1^3 && a_{1,2}\lambda_2^2 && \lambda_3^2a_{1,3} && \lambda_4^2a_{1,2} \\
&&&\\
\lambda_1^2a_{2,1}  && \lambda_1 \lambda_2^2 && \lambda_3 a_{2,1}a_{1,3} && a_{1,3}a_{3,4}a_{4,1} \\
&&&\\
\lambda_1 a_{3,4}a_{4,1}  && a_{3,4}a_{4,1}a_{1,2} && \lambda_3a_{1,2}a_{2,1} && a_{3,4}a_{4,1}a_{1,2}\\
&&&\\
\lambda_1^2a_{4,1} && \lambda_2 a_{4,1}a_{1,2} && \lambda_3 a_{4,1}a_{1,3} && \lambda_4a_{4,1} a_{1,2}
\end{array}
\right).$$

The determinant is attained by $$\lambda_1^3 (\lambda_1\lambda_2^2 )
(a_{3,4} a_{4,1}a_{1,2}) (\lambda_3  a_{4,1}a_{1,3})\text{ and
}\lambda_1^3(a_{1,3}a_{3,4}a_{4,1})( \lambda_3a_{1,2}a_{2,1}
)(\lambda_2 a_{4,1}a_{1,2}),$$ where all elements are identical,
and~${\lambda_1\lambda_2}={a_{1,2}a_{2,1}}.$ That is, the ghost
determinant is not an occasional outcome of repeated values  (such
as~$9,10$ in the entries of~$A$), or some relations between
coefficients. The singularity which we encounter is systematic:
  $$\lambda_1^3 (\underbrace{\lambda_1 \lambda_2^2}_{\alpha\lambda_2})
(a_{4,1}a_{1,2}a_{3,4}) (a_{4,1}a_{1,3}\lambda_3 )=\lambda_1^3[\lambda_2
a_{4,1}a_{1,2}][a_{3,4}a_{4,1}a_{1,3}][a_{1,2}a_{2,1}\lambda_3 ].$$

\subsection{Resolving the pathology}

In this section we offer sufficient conditions for independence, and
present two conjectures on the eigenvectors of the quasi-inverse of
a matrix.

\subsubsection{The resolution by means of disjoint index sets}

The intersection of the $\{I_\lambda\}$ causes the eigenvector
dependency seen in the previous section. This pathology will be
resolved in the following theorem using disjoint $\{I_\lambda\}$, in
which we show that it is a Zariski-closed condition.

\begin{thm}\label{es-n} Let $A=(a_{i,j})$ be a   nonsingular $n\times n$ matrix, with tangible characteristic polynomial (coefficient-wise) and $n$ distinct eigenvalues. If $A$ satisfies
the   difference criterion, then the eigenvectors of $A$ are
tropically independent. \end{thm}

\begin{proof}
\noindent
Let~$f_A(x)=\sum_{i=0}^n\alpha_ix^{n-i}\in\mathcal{T}[x]$
be the characteristic polynomial of $A$, which means
$\alpha_0=0,\ \alpha_1=\tr(A),\ \alpha_n=\det(A).$ Without loss of generality,~$\tr(A)=a_{1,1}$, i.e.,
\begin{equation}\lambda_1=\label{n-1}a_{1,1}>a_{t,t}\ \forall t\ne 1.\end{equation}
In order to get $n$ distinct eigenvalues, we must have
$f_A(x)=f^{es}_A(x)$, or equivalently
\begin{equation}\label{n-2}\lambda_1=\tr(A)>\lambda_2=\frac{\alpha_2}{\tr(A)}>\lambda_3=\frac{\alpha_3}{\alpha_2}>\dots>\lambda_{n-1}=\frac{\alpha_{n-1}}{\alpha_{n-2}}>\lambda_n=\frac{\det(A)}{\alpha_{n-1}},\end{equation}
where $\{\lambda_l\}_{l\in[n]}$ are the corner-roots of $f_A$.
Otherwise,~$\exists t,s$ such that~$\ f_A(\lambda_t)\in \mathcal T$
or~$\lambda_t=\lambda_s$,  contrary to hypothesis. In particular, $\tr(A)=\lambda_1$ and $\Ind_1=\{1\}.$

\vskip 0.25 truecm

We need to show that $\det(W) \in\mathcal{T},$ where~$W$ is the
matrix of eigenvectors. This is achieved in three steps:
\begin{enumerate}
\item For every~$k\in[n]$, $\Ind_k\subseteq\Ind_{k+1}\ \forall k$, and therefore $I_{\lambda_k}=\{k\}$.

\item For every~$k\in[n]$,
$W_{k,k}=\lambda_1\cdots\lambda_{k-1}\lambda_k^{n-k}$.
\item Finally, $\det(W)=\prod_{k\in[n]}W_{k,k}\in\mathcal{T},$ as desired.
\end{enumerate}

\vskip 0.25 truecm

\noindent \textbf{Step (1).}  A straightforward application of
\eqref{n-2} yields
\begin{equation}\label{n-3}\lambda_1\cdots\lambda_k=\alpha_k\in\mathcal{T}\ \text{and}\ \{1\}=\Ind_1\subseteq\Ind_k,\ \forall k\geq 1.\end{equation}  Otherwise,  $a_{1,1}\cdot \alpha_{k-1}$ would yield a permutation on $k$ indices,
dominated by $\alpha_k$:
$$\lambda_1=a_{1,1}\leq\ \underbrace{\overbrace{\frac{\alpha_{_{k}}}{\alpha_{_{k-1}}\cdot a_{1,1}}}}_{\geq 1_R}\ \cdot a_{1,1}=\lambda_k,\ \text{  contradicting \eqref{n-2}.}$$

Let $\Ind_k=\{1,j_2,...,j_k\}$. Since $\Ind_0=\emptyset,$ there
exists $i\leq k:\ j_s\in I_{\lambda_i},\ \forall s\in\{2,...,k\}.$
Assume that $\Ind_{l-1}\subseteq\Ind_{l}$ holds  through  $l = k$,
and then fails for $k+1$. That is,
$$\forall l\leq k\ \ \Ind_{l-1}\subseteq \Ind_l\ \text{ and }\ \exists s\in\{2,..., k\}:\ j_s\notin \Ind_{k+1}.$$
However, since~$\Ind_n=[n]$,~$j_s\in\Ind_n$. We define~$t$ to be the
minimal index~$k<t<n$ such that $j_s\in \Ind_t\text{ but }
j_s\notin\Ind_{t-1}.$ That is $j_s\in I_{\lambda_t}\cap
I_{\lambda_i}$ for some $i,t:\ i<k<t$, contradicting  the difference
criterion. Therefore $\Ind_k\subseteq\Ind_{k+1},\ \forall k$. \vskip
0.25 truecm

\noindent \textbf{Step (2).}   Up to some permutation, we may require w.l.g.~that~$j_k=k,\ \ \forall k\in[n]$. That is, $I_{\lambda_k}=\{k\}\ \forall k\in[n]$,
\begin{equation}\label{m-2}a_{1,1}>\lambda_k=\ \underbrace{\overbrace{\frac{\alpha_k}{\alpha_{k-1}\cdot a_{t,t}}}}_{\geq 1_R}\ \cdot a_{t,t}\geq a_{t,t},\ \ \forall t\geq k,\ \forall k>1,\end{equation}
with equality only when $k=t$, and
\begin{equation}\label{m-3}\beta_{k,t}=\max\{\lambda_k,a_{t,t}\},\ \ \forall t< k.\end{equation}
 Thus, the entries of the $k^{th}$ eigenmatrix $A+\lambda_kI=\big(b^{(k)}_{i,j}\big)$ are given by
\begin{equation}\label{diagonal}b^{(k)}_{i,j}=\begin{cases}\lambda_1&,\  i=j=1\\\beta_{k,i}&,\  1<i=j<k\\ \lambda_k&,\ i=j\geq k\\a_{i,j}&,\  i\ne j\end{cases}.\end{equation}
 (For example, for $k=2$ and $k=3$ we get
$$\left(\begin{array}{cccc}
\lambda_1&&a_{i,j}:i<j& \\
   &\lambda_2&&\\
&&\ddots&\\
   &a_{i,j}:i>j & &  \lambda_2
\end{array}
\right),\
\left(\begin{array}{ccccc}
\lambda_1& &a_{i,j}:i<j&\\
   &\beta_{3,2}&&&\\
&&\lambda_3&&\\
&&&\ddots & \\
&a_{i,j}:i>j&&  & \lambda_3
\end{array}
\right),$$ respectively, where $a_{i,j}$ indicates that the off-diagonal entries are identical to those of~$A$.)

\vskip 0.3 truecm

Let  $\adj(A)=(a'_{i,j})$, $W=(w_{i,j})$ be the matrix with the (tangible value of the) eigenvectors for its columns, and notice that $w_{k,k}=\adj(A+\lambda_k I)_{k,k}.$

\begin{itemize}
\item On  one hand, by~\eqref{diagonal} $(A+\lambda_k I)_{k,k}=\lambda_k$. By ~\cite[Theorem 2.8]{STMA2}, \begin{equation}\label{attain}\big((A+\lambda_k I)\adj(A+\lambda_k I)\big)_{k,k}=\det(A+\lambda_k I)=f_A(\lambda_k)\in \mathcal{G},\end{equation} where $f_A(\lambda_k)=\alpha_k\lambda_k^{n-k}+\alpha_{k-1}\lambda_k^{n-k+1}=(\lambda_1\cdots\lambda_{k-1}\lambda_k^{n-k+1})^\nu,$ as $\lambda_k$ is the $k^{th}$ corner root of the polynomial of distinct coefficients $f_A$. Since every summand in~\eqref{attain} is dominated by this expression, we get
$$\adj(A+\lambda_k I)_{k,k}=w_{k,k}\leq {\lambda_1\cdots\lambda_{k-1}\lambda_k^{n-k}}.$$

\item On the other hand,~$\lambda_k^{n-k}\det(M)$ is a summand in~$\adj(A+\lambda_k I)_{k,k},$ where~$M$ is the~$(k-1)\times (k-1)$-principal sub-matrix of~$A+\lambda_k I,$ obtained by rows and  columns~$[k-1]$.  Since~$A\leq A+\lambda_kI$ entry-wise (and in particular for~$M$ and its corresponding principal sub-matrix in~$A$), we get~$\det(M)\geq_{\nu}\alpha_{k-1}=\lambda_1\cdots \lambda_{k-1}.$ Thus,
$\adj(A+\lambda_k I)_{k,k}=w_{k,k}\geq\lambda_1\cdots\lambda_{k-1}\lambda_k^{n-k}.$\end{itemize}

 As a result,
\begin{equation}\label{W}
w_{k,k}=\lambda_1\cdots\lambda_{k-1}\lambda_k^{n-k},\ \forall k\in[n].
\end{equation}

\vskip 0.25 truecm

\noindent \textbf{Step (3).} Notice that
\begin{equation}\label{detW}\prod_{k\in[n]}w_{k,k}=\prod_{k\in[n]}
\lambda_1\cdots\lambda_{k-1}\lambda_k^{n-k}=\prod_{k\in[n-1]}
\lambda_1\cdots\lambda_{k-1}\lambda_{k}^{n-k+1}.\end{equation} We
claim that any other permutation in~$W$ is strictly dominated by the
term in~\eqref{detW}.

Let~$X=(x_{i,j})$ be an~$n\times n$ matrix. For~$\pi\in S_n$ denote~$X_\pi=\prod_{i\in[n]}x_{i,\pi(i)}$, and its cycles are referred to as~$X$-cycles. An~$X$-cycle of length~$d$ is said to be  an~$X^{(d)}$-cycle. Using~\eqref{diagonal}, we denote by~$W_\pi^{(\lambda)}$ the product of  eigenvalues of~$A$ in~$W_\pi$. For example~$W_\Id^{(\lambda)}=W_\Id.$

For every~$\pi\ne\Id,\ \ W_\pi=W_\pi^{(\lambda)}\cdot C$, where~$C$ is a product of~$A-$cycles, and~$W_\pi^{(\lambda)}$ is a product in~$W_\Id$. A cycle~$c\in C$ is  an~$X^{(d)}$-cycle for some~$d\in[n]$, and is dominated by~$\lambda_1\cdots \lambda_d$, which is strictly dominated by~$\lambda_1\cdots \lambda_{d-2}\lambda_{d-1}^2<\lambda_1\cdots \lambda_{d-3}\lambda_{d-2}^3<...<\lambda_1^d$. Therefore,~$W_\pi\leq W_\Id$, and we show strict dominance. The product~$W_\pi=W_\pi^{(\lambda)}\cdot C$ satisfies at least one of the following cases:

\begin{itemize}
\item $C$ includes an~$A^{(n)}-$cycle, dominated by~$\lambda_1\cdots \lambda_n$, which is strictly dominated by~$\lambda_1\cdots \lambda_{n-2}\lambda_{n-1}^2$ in~\eqref{detW}.

\item $C$ includes two different~$A^{(d)}-$cycles, at least one is strictly dominated by~$\lambda_1\cdots \lambda_{d}$ in~\eqref{detW}.

\item $C$ includes an~$A^{(d)}-$cycle which does not act on some index of~$[d]$, making it strictly dominated by~$\lambda_1\cdots \lambda_{d}$ in~\eqref{detW}.

\item $C=\prod c,$ s.t.~$c$ is an~$A^{(d_c)}-$cycle on indices~$[d_c]$. Then,  $$W_\pi^{(\lambda)}=\prod_{j\in J\subseteq [n]}\lambda_j\ \ \Rightarrow\ \ \exists j\in J:\ j\ne 1,$$
whereby~$W_\pi^{(\lambda)}$ is strictly dominated by~$\lambda_1^m$
in~\eqref{detW}, for~$m=|J|$.
\end{itemize}
Since at least one term is strictly dominated, and the rest are dominated, the assertion  follows.

\end{proof}

\subsubsection{The resolution by means of the quasi-inverse }

In view of the results in~\cite{EP},~\cite{PI&CP} and~\cite{MPDM},
one can conclude that  quasi-inverse matrices play an important role
in formulating properties of matrices. These studies lead us to the
following two conjectures, based on a    further examination of
Example~\ref{exa}.

\begin{con}\label{esnab1} Let $A$ be a nonsingular matrix with $n$ distinct eigenvalues. If the eigenvectors of $A$ are dependent, then
\begin{enumerate}
\item Recalling Theorem~\ref{CPIP},~$\det(A)f_{A^\nabla}(x)$ strictly ghost-surpasses~$x^nf_A(x^{-1})$.
\vskip 0.2 truecm

\item The matrix~$A^\nabla$ has fewer distinct eigenvalues than $A$, when~$f_{A^\nabla}\ne f^{^{es}}_{A^\nabla}.$
\vskip 0.2 truecm

\item Moreover, the eigenvectors of~$A^\nabla$ are independent.
\end{enumerate}
\end{con}

\begin{con}\label{esnab2} Let $A$ be a nonsingular matrix.
If $A^\nabla$ has $n$ distinct eigenvalues, then their corresponding eigenvectors are independent.
\end{con}

Let us consider Conjecture~\ref{esnab1} in the case of
Example~\ref{exa}. We recall Theorem~\ref{detadj} and
Lemma~\ref{invnab}, to conclude that $\det(\adj(A))=\det(A)^{n-1}$
is attained solely by the permutation $\sigma^{-1}$, where $\det(A)$
is attained solely by $\sigma$.

Let $A$ be as in Example~\ref{exa}. As a result
$$\adj(A)=
\left(
\begin{array}{cccc}
- & -& -& 19\\
-& 27 & -& 27\\
19 & 28 & - & 28\\
- & - & 19 & -
 \end{array}
\right),$$$$f_{\adj(A)}(x)=x^4+27x^3+47x^2+74^\nu x+84,$$$$\text{and }\Ind_1=\{2\},\ \Ind_2=\{3,4\},\ \Ind_3=\{2,4,3\}=\{2,3,4\},\ \Ind_4=\{3,1,4,2\}.$$
 (Indeed~$\det(\adj(A))=\det(A)^{4-1}$, and~$f_{A^\nabla}$ is obtained by coefficients~$\frac{\alpha_k}{\det(A)^k}$.)

It is easy to see that
$$\Ind_1\setminus\emptyset=\{2\},\ \Ind_2\setminus\Ind_1=\{3,4\},\ \Ind_3\setminus\Ind_2=\{2\},\ \Ind_4\setminus\Ind_3=\{1\}$$
are not disjoint. However, calculating the eigenvalues of~$\adj(A)$  reveals these are not the sets~$I_{\lambda_k}$. That is,~$f^{es}_{_{\adj(A)}}(x)=x^4+27x^3+74^\nu x+84,$ and the dependence in the principal sub-matrix of~$\{2,3,4\}$  (identical to the minor causing dependence  in~$W$),
$$\lambda_2(\lambda_1\lambda_2)(a_{4,1}a_{1,2}a_{3,4})(a_{4,1}a_{1,3})\lambda_3=\lambda_2(a_{4,1}a_{1,2})(a_{3,4}a_{4,1}a_{1,3})(a_{1,2}a_{2,1})\lambda_3\Rightarrow$$
 $$(\lambda_1\lambda_2)(a_{4,1}a_{1,2}a_{3,4})(a_{4,1}a_{1,3})\frac{a_{4,1}a_{1,3}a_{3,4}}{a_{1,1}}=(a_{4,1}a_{1,2})(a_{3,4}a_{4,1}a_{1,3})(a_{1,2}a_{2,1})\frac{a_{4,1}a_{1,3}a_{3,4}}{a_{1,1}},$$
increases the coefficient of~$x$,  causing~$47x^2$  to be inessential.
As a result,
$$I_{\lambda_1}=\{2\},\ I_{\lambda_{2,3}}=\{3,4\},\ I_{\lambda_4}=\{1\},$$ where $\lambda_1=27,\ \lambda_{2,3}=23.5\text{ (with multiplicity 2),}\text{ and}\ \lambda_4=10.$
As the conjecture predicted, the  eigenvectors
$$v_1=(66,81,82,74)=66(0,15,16,8),\ v_4=(74,65,55,65)=55(19,10,0,10),$$
$$\text{and } v_{2,3}=65^{-1}\underbrace{(65, 69.5, 74, 69.5)}_\text{from the third column}=(0,4.5,9,4.5)=69.5^{-1}\underbrace{(69.5,74,78.5,74)}_\text{from the fourth column}$$ are independent.

\subsubsection{The  resolution by means of generalized
eigenspaces}\label{geneig}

\textbf{Eigenspaces} are studied in ~\cite{STMA2} and are defined
in~\cite{STMA3} to be  spanned by   supertropical eigenvectors. Let
$V = F^n.$

\begin{df} A tangible vector $v\in V$ is a \textbf{generalized supertropical
eigenvector} of~$A$, with \textbf{generalized supertropical
eigenvalue} $\lambda\in \mathcal T$, if $(A + \lambda I)^m v$ is
ghost for some $m\in \mathbb{N}$. If $A  ^m v$ is itself ghost for
some $m$, we call the generalized eigenvector
 $v$ \textbf{degenerate}.

 The minimal such $m$ is
called the \textbf{multiplicity} of the eigenvalue (and also of the
eigenvector).

The \textbf{generalized supertropical eigenspace} $V_\lambda$ with
\textbf{generalized supertropical eigenvalue} $\lambda\in \mathcal T$
is the set of generalized supertropical eigenvectors with
generalized supertropical eigenvalue $\lambda$.
\end{df}

Note that if $v$ is a degenerate eigenvector, then it belongs to
$V_\lambda$ for all sufficiently small $\lambda$.

\begin{lem} $V_\lambda$ is indeed a supertropical subspace of $V$. \end{lem}

\begin{proof} Let $v,u\in V_\lambda.$ Thus $\exists m,t:\ (A+\lambda I)^mv\models_{gs} 0_{\mathcal F}\ $ and $\ (A+\lambda I)^tu\models_{gs} 0_{\mathcal F}$, and therefore for any~$a\in {\mathcal F}$ $$(A+\lambda I)^{(m+t)}(v+au)=(A+\lambda I)^t(A+\lambda I)^mv+a(A+\lambda I)^m(A+\lambda I)^tu\models_{gs} 0_{\mathcal F}.$$
\end{proof}

\begin{rem}\label{hierarchy}
We have the following hierarchy:
$$Av\models_{gs} \lambda v,\ \text{implies}\ A^mv \models_{gs} \lambda^mv,\ \text{implies}\ A^mv + \lambda ^m v\models_{gs}0_{\mathcal F}, \ \text{implies}\ (A
+ \lambda I)^m v\models_{gs}0_{\mathcal F}.$$

\end{rem}

This approach gives some insight into the difference criterion. For
the remainder of this paper we use the well-known digraph of a
matrix, whose vertices are the indices $\{1, \dots, n\}$ and whose
edges correspond to the nonzero entries $a_{i,j}$ of the matrix. Any
permutation $\pi$ corresponds to some    cycle of length $n$ which
can be decomposed into disjoint simple cycles, and the contribution
of the permutation to the determinant is the product of their
weights.  For any cycle of length $k$ and weight $\mu$, its $k$-th
power lies on the diagonal with all of the entries equal to $\mu$
(so that its weight is $\mu^k$).  Thus, the corresponding part of
the diagonal of $A^k$ (and all subsequent powers) dominates all
$k$-th powers cycles of length $k$, and in particular this is the
case for $A^{m} = (A^k)^{m/k},$ for any multiple $m$ of $n!$.

The
diagonal is a dominant permutation  of $A^m$.

\begin{lem}  The difference criterion is satisfied
for $A$ iff the diagonal entries of $A^m$ are distinct, whenever
$n!$ divides $m$.
\end{lem}
\begin{proof} $(\Rightarrow)$ The
diagonal is a dominant permutation  of $A^m$.  The difference
criterion implies that all of these diagonal entries are
distinct.

$(\Leftarrow)$ Suppose that in $A^m$ some index $i$ appears in both
$I_k$ and $I_{k'}$ for $k<k'$, where $k$ is taken minimal such. Then
all the previous $I_j$ are disjoint, so, rearranging the diagonal
entries, we may assume that $i$ appears in the $|I_1| + \dots +
|I_{k-1}| +\alpha$ position in the diagonal for some $1 \le \alpha
\le |I_{k}|.$ But $i$ must also appear in the analogous position
arising from $I_{k'},$ for some $k'>k$, so $A^m$ has a double diagonal entry.
\end{proof}

\begin{lem}\label{es-n1} If $A$ is nonsingular and diagonally dominant, then the diagonal of $A$ is tangible.
\end{lem}
\begin{proof} The determinant is the product of the diagonal
entries, so each is tangible.
\end{proof}
 In view of Remark~\ref{eigen1}, we can refine the generalized
supertropical eigenspaces $V_\lambda$. Write $f_A= \prod _i g_i$
where  $g_i = (x + \lambda_i)^{t_i},$ with the $\lambda _i$
distinct, and let $\tilde f_i = \prod _{j\ne i} g_j.$ (Thus, $f _A =
g_i \tilde f_i.$) Suppose $v \in \tilde f_i(A)V_\lambda.$ Then
$g_i(A) v \in f_A(A) V _\lambda$ is ghost, implying $ v \in
V_\lambda$. Thus, we  can define the subspace
$$V'_{\lambda_i} = \left(\prod _{j\ne i} g_j(A)\right) V ,$$
which is a generalized supertropical eigenspace  with respect to
$\lambda_i$.

 \begin{df} A matrix $A$ is \textbf{strongly nonsingular} if  $A^m$  is nonsingular for all~$m$.\end{df}

\begin{lem}\label{nonon} A strongly nonsingular matrix $A$ has no nonzero degenerate generalized
eigenvectors.\end{lem}
\begin{proof}
Take $m$ large enough (say $n!$) such that $A^m $ is dominated by
the diagonal. Write $A^m = (a_{i,j})$ and $v = (v_1, \dots, v_n)$.
Then we have a contradiction to $A^m v\in (\mathcal G\cup
\{0_{\mathcal F}\})^n$ unless for each $i$ there is $i' = f(i)$ such
that $a_{i,i'}v_{i'} \ge a_{i,i}v_i.$  Write $f^1 = f$ and $f^k =
f(f^{k-1}),$ and $a_k = a_{f^k(i), f^{k-1}(i)}.$ Then $f^k(i) =
f^{k+t}(i)$ for  $t\le n$, and $a_{k+t}\dots a_t \ge 1,$
contradicting $A^{n!}$ nonsingular (since the dominant path is on
the diagonal).
\end{proof}

\begin{thm}\label{es-n2} If $A$ is strongly nonsingular, then the $V_{\lambda_i}'$ are independent
\end{thm}
\begin{proof}
We can replace $A$ by $A^{n!}$ and assume that $A$ is diagonally
dominant and that $V_{\lambda_i} '$ are eigenspaces of $A$. We use
the notation following Lemma~\ref{es-n1}.

 We assume on the contrary that we have a ghost dependence,
i.e., $\sum _{i\in[u]} \gamma _i \tilde f_i(A) v_i $ ghost for
tangible $\gamma_i$, and aim for a contradiction. Since $A$ is
strongly nonsingular, the $g_j$ act like scalar multiplication by
$\lambda_j,$ in view of Lemma~\ref{nonon}, and, furthermore,
$\lambda_u ^{t_u}$ dominates all $ \lambda _u^{t_u-j}\beta ^j$, for
all $\beta < \lambda_u.$ Hence, when $x$ is to be specialized to
these $\beta,$ $\lambda_u ^{t_u}$ dominates $\sum _j \lambda
_u^{t_u-j}x _u^j = g_u$, and thus, by the argument of
Lemma~\ref{nonon},
 some component
of $\gamma _u \lambda_u ^{t_u}\tilde f_u(A) v_u$ is dominant in
$\gamma _u g_u(A)\tilde f_u(A) v_u$, a ghost. Therefore  some power
of $A$ ghost
 annihilates $ \tilde f_u(A) v_u = (\prod _{u'\ne u} \lambda_{u'})v_u$,  contradicting  $A$ being strongly
nonsingular.
\end{proof}


\begin{thebibliography}{xx}
 \newcommand{\au}{\sc}
 \newcommand{\ti}{\it}
\vskip 0.25 truecm
\normalsize







\bibitem{MPA}
{\au M.~Akian, R.~Bapat,  S.~Gaubert},
{\ti Max-plus algebra}.
Hogben L., Brualdi R., Greenbaum A., Mathias R. (eds.), Handbook of Linear Algebra. Chapman and Hall, London, 2006.
\vskip 0.25 truecm

\bibitem{LDTS}
{\au M. ~Akian, S. ~Gaubert, A. ~Guterman},
{\ti Linear independence over tropical semirings and beyond}.
Tropical and Idempotent Mathematics,   Contemporary Math.~495, 1--38, AMS, 2009.
\vskip 0.25 truecm

\bibitem{polyhedra}
{\au M. ~Akian, S. ~Gaubert, A. ~Guterman},
{\ti Tropical polyhedra are equivalent to mean payoff games}.
J. Algebra Comput. ~22(1), 1250001-1--1250001-43, 2012.
\vskip 0.25 truecm


\bibitem{DMP}
{\au M. ~Akian, S. ~Gaubert, C. ~Walsh},
{\ti Discrete max-plus spectral theory}.
Idem. ~Math. ~and Mathematical Physics, G.L. Litvinov and V.P. Maslov (eds.),  Contemporary Math.  ~377, 53--77, AMS, 2005.
\vskip 0.25 truecm




\bibitem{TB}
{\au M. ~Akian, S. ~Gaubert, A. ~Marchesini},
{\ti Tropical bounds for eigenvalues of matrices}.
J. ~Linear Algebra Appl. ~446, 281--303, 2014.


\vskip 0.25 truecm
\bibitem{EP}
{\au M.~Akian, S.~Gaubert, A.~Niv},
{\ti Tropical compound matrix identities}.
Preprint, 2015.

\vskip 0.25 truecm


\bibitem{CCP}
{\au P.~Butkovic},
{\ti On the coefficients of the max-algebraic characteristic polynomial and equation}.
In proceedings of the workshop on Max-algebra, Symposium of the International Federation of Automatic Control, Prague, 2001.
\vskip 0.25 truecm

\bibitem{MA}
{\au P.~Butkovic},
{\ti Max-algebra: the linear algebra of combinatorics?}.
J. ~Linear Algebra Appl. ~367, 313--335, 2003.
\vskip 0.25 truecm

\bibitem{ETCP}
{\au P.~Butkovic, L.~Murfitt},
{\ti Calculating essential terms of a characteristic maxpolynomial}.
CEJOR 8, 237--246, 2000.
\vskip 0.25 truecm






\bibitem{LOP}
{\au M.~Fiedler, J.~Nedoma, J.~Ramik, J.~Rohn, K.~Zimmermann},
{\ti Linear optimization problems with inexact data}.
Springer, New York, 2006.
\vskip 0.25 truecm




\bibitem{G.Ths}
{\au S.~Gaubert},
{\ti Th\'{e}orie des syst\`{e}mes lin\'{e}aires dans les dio\"{i}des}.
PhD dissertation, School of Mines. Paris, July  1992.
\vskip 0.25 truecm



\bibitem{TSPM}
{\au S.~Gaubert, M. ~Sharify},
{\ti Tropical Scaling of Polynomial Matrices}.
Lecture Notes in Control and Information Sciences, no.~389,  291--303, Springer, 2009.
\vskip 0.25 truecm

\bibitem{PA}
{\au M.~Gondran},
{\ti Path algebra and algorithms}.
In B. ~Roy, editor, Combinatorial programing: methods and applications, Reidel, Dordrecht, ~137--148, 1975.
\vskip 0.25 truecm



\bibitem{TAG}
{\au I.~Itenberg, G.~Mikhalkin, E.~Shustin},
{\ti Tropical Algebraic Geometry}.
Oberwolfach Seminars  35, Birkhauser, Basel, 2007.
\vskip 0.25 truecm

\bibitem{TA}
{\au Z.~Izhakian},
{\ti Tropical arithmetic and matrix algebra}.
Comm.~Algebra  37(4),1445--1468, 2009.
\vskip 0.25 truecm

\bibitem{ST&SV}
{\au Z.~Izhakian, M.~Knebusch, L.~Rowen},
{\ti Supertropical semirings and supervaluations}.
J. ~Pure and App. ~Alg.  ~215(10), 2431--2463,  2011.
\vskip 0.25 truecm

\bibitem{STVT}
{\au Z.~Izhakian, M.~Knebusch, L.~Rowen},
{\ti A Glimpse at Supertropical Valuation Theory}.
J. ~An. ~Ştiinţ. ~Univ. ``Ovidius" Constanţa, Ser. ~Mat. ~19(2), 131--142, 2011.
\vskip 0.25 truecm



\bibitem{STLA}
{\au Z.~Izhakian, M.~Knebusch, L.~Rowen},
{\ti Supertropical linear algebra}.
Pacific Journal of Mathematics  266(1), 43--75, 2013.
\vskip 0.25 truecm



\bibitem{STA}
{\au Z.~Izhakian, L.~Rowen},
{\ti Supertropical algebra}.
Adv. ~Math. ~225, 2222--2286, 2010.
\vskip 0.25 truecm

\bibitem{STMA}
{\au Z. \ Izhakian, L.\ Rowen},
{\ti Supertropical matrix algebra}.
Israel  Math. ~182(1), 383--424, 2011.
\vskip 0.25 truecm

\bibitem{STMA2}
{\au Z.~Izhakian, L.~Rowen},
{\ti Supertropical matrix algebra II: solving tropical equations}.
Israel  Math. 186(1), 69--97, 2011.
\vskip 0.25 truecm

\bibitem{STMA3}
{\au Z.~Izhakian, L.~Rowen},
{\ti Supertropical matrix algebra III: Powers of matrices and their supertropical eigenvalues}.
J. ~Algebra~341(1), 125--149, 2011.
\vskip 0.25 truecm





\bibitem{APP}
{\au J.~Jonczy},
{\ti Algebraic Path Problems}.
RUN Seminar, University of Berne, Switzerland. 2008
\vskip 0.25 truecm

\bibitem{IM}
{\au G. L.~Litvinov, V. P.~Maslov},
{\ti Idempotent mathematics: correspondence principle and applications}.
Russian Math.~Surveys 51(6),  1210--1211, 1996.
\vskip 0.25 truecm


\bibitem{T&I}
{\au G. L.~Litvinov, S.N.~Sergeev},
{\ti Tropical and Idempotent Mathenatics}.
Contemporary Math.~495. AMS, Providence, 2009.
\vskip 0.25 truecm


\bibitem{TGA}
{\au  G.~Mikhalkin},
{\ti Tropical geometry and its applications}.
In Proceedings of the ICM, Madrid, Spain, vol. ~II, ~827-852. (arXiv:  [math. AG]0601041v2), 2006
\vskip 0.25 truecm

\bibitem{CP}
{\au  A.~Niv},
{\ti Characteristic Polynomials of Supertropical Matrices}.
Comm.~Algebra  ~42(2), 528--539, 2014.
\vskip 0.25 truecm


\bibitem{PI&CP}
{\au  A.~Niv},
{\ti On Pseudo-inverses of matrices and their characteristic polynomials in supertropical algebra}.
Linear Algebra  Appl., to appear. (arXiv:1306.5861).
\vskip 0.25 truecm






\bibitem{MPDM}
{\au  S.N.~Sergeev}, {\ti Max-plus definite matrix closures and
their eigenspaces}. Linear Algebra  Appl. ~421, 182--201, 2007.
\vskip 0.25 truecm








\bibitem{CH}
{\au H.~Straubing},
{\ti A combinatorial proof of the Cayley-Hamilton Theorem}.
Discrete Math. ~43(2-3), 273--279, 1983.
\vskip 0.25 truecm






\end{thebibliography}
\end{document}